\documentclass{LMCS}

\def\dOi{11(1:17)2015}
\lmcsheading%
{\dOi}
{1--19}
{}
{}
{May\phantom.~16, 2013}
{Mar.~31, 2015}
{}

\ACMCCS{[{\bf Mathematics of computing}]: Continuous
  mathematics---Topology---Point-set topology}

\usepackage{amssymb,amsmath,mathabx}
\usepackage{graphicx,color}
\usepackage{hyperref}

\usepackage{bbm}  
\def\2{\mathbbm{2}}   

\def\bd{\mathop{\mathrm{bd}}\nolimits}
\def\cl{\mathop{\mathrm{cl}}\nolimits}
\def\ext{\mathop{\mathrm{ext}}\nolimits}
\def\height{\mathop{\mathrm{height}}\nolimits}
\def\ind{\mathop{\mathrm{ind}}\nolimits}

\newcommand{\sse}{\sqsubseteq}
\newcommand{\spe}{\sqsupseteq}
\def\Idl{\mathop{\mathrm{Idl}}\nolimits}
\def\succ{\mathop{\mathrm{succ}}\nolimits}
\newcommand{\lub}{\sqcup}
\newcommand{\upa}[1]{{\uparrow\hspace{-3pt} #1}}
\newcommand{\down}[1]{{\downarrow\hspace{-3pt} #1}}
\newcommand{\comp}{\uparrow}
\newcommand{\ncomp}{{\ooalign{{\hss{\scalebox{1}[0.6]{/}}\hss}\crcr{$\ \uparrow\ $}}}}

\def\BB{\mathfrak{B}}

\def\N{\mathbb{N}}
\def\T{\mathbb{T}}
\def\I{\mathbb{I}}

\newcommand{\wh}[1]{\widehat{#1}}
\def\dom{\mathop{\mathrm{dom}}\nolimits}

\newcommand{\sigbot} {\mathbb{T}^{\omega}}

\def\exS{S_{\mathrm{ex}}}
\def\exbarS{\bar{S}_{\mathrm{ex}}}

\def\deg{\mathop{\mathrm{deg}}\nolimits}
\def\level{\mathop{\mathrm{level}}\nolimits}

\newcommand{\fk}[1]{\mathfrak{#1}}
\def\nn{n \in \N}
\def\a2{a \in \2}

\begin{document}
\title{Domain Representations Induced by Dyadic Subbases}

\author{Hideki Tsuiki}
\address{Graduate School of Human and Environmental Studies, Kyoto University, Sakyo-ku, Kyoto, Japan}
\email{\{tsuiki,tsukamoto\},@i.h.kyoto-u.ac.jp}
\thanks{This work was partially supported by JSPS KAKENHI Grant Number 22500014.} 

\author{Yasuyuki Tsukamoto}
\address{\vspace{-18 pt}}

\makeatletter 
\@namedef{subjclassname@1998}{2012 ACM CCS}
\makeatother

\keywords{domain theory, subbase, compact Hausdorff space}


\begin{abstract}
  We study domain representations induced by dyadic subbases and
  show that a proper dyadic subbase of a second-countable
  regular space $X$ induces an embedding of $X$ in the
  set of minimal limit elements of 
  a subdomain $D$ of $\{0,1,\bot\}^\omega$.  In particular,
  if $X$ is compact, then 
  $X$ is a retract of the set of limit elements of $D$.


\end{abstract}

\maketitle

\section{Introduction}

From a computational point of view, it is natural to consider a subbase
of a second-countable $T_{0}$ space $X$
as a collection of primitive properties of $X$
through which one can identify each point of $X$.
In this way, by fixing a numbering of the subbase,
one can represent each point of $X$ as a subset of $\N$ and
construct a domain representation of $X$
in the domain $P_{\omega}$ of subsets of $\N$
\cite{Blanck:2000,StoTucker08}. 
Note that
$P_{\omega}$ is isomorphic to the domain of infinite sequences
of the Sierpinski space $\{1, \bot\}$.

On the other hand, 
each regular open set $A$
(i.e., an open set which is equal to the interior of its closure) 
of a topological space $X$
divides $X$ into three parts:
$A$, the exterior of $A$, and their common boundary.
Therefore, one can consider
a pair of regular open subsets which are exteriors of each other 
as a pair of primitive properties
and use a subbase which is composed of such pairs of open sets
in representing the space.
Such a subbase is called a dyadic subbase
and a dyadic subbase of a space $X$ induces a domain representation of $X$
in the domain $\T^\omega$ of infinite sequences of $\T = \{0,1,\bot\}$. 
In \cite{Tsuiki:2004b} and \cite{OTY},
the authors introduced to a dyadic subbase the properness property 
which expresses a kind of orthogonality between the components
and 
studied domain representations of Hausdorff spaces 
induced by proper dyadic subbases.
In this representation,
the domain is fixed to $\T^{\omega}$
and an embedding $\varphi_{S}$ of a Hausdorff space $X$ in $\T^{\omega}$ is
derived from a proper dyadic subbase $S$ of $X$.

In this paper, we derive from a dyadic subbase $S$
a domain (i.e., an $\omega$-algebraic pointed dcpo) ${D}_S$ and  a bounded complete domain
$\wh{D}_S$ which are subdomains of $\T^{\omega}$ containing 
$\varphi_{S}(X)$ as subspaces. 
The domain ${D}_S$ has the following properties.
(1) If $X$ is a strongly nonadhesive Hausdorff space 
(Definition \ref{def-ad}), then
the set $L({D}_S)$ of limit (i.e., non-compact) elements of ${D}_S$
has minimal elements.
(2) If $X$ is regular, then $\varphi_{S}$ is an embedding of $X$ 
in the set of minimal elements of $L({D}_S)$.
(3) If $X$ is compact, then there is a retraction $\rho_{S}$ from $L({D}_{S})$ to $X$.
That is, every infinite strictly increasing sequence in $K({D}_{S})$ represents a point of $X$ through $\rho_{S}$ and
$ ({D}_{S}, L({D}_{S}), \rho_{S})$ is the kind of domain representations
studied in \cite{Tsuiki:2004}.
The domain $\wh{D}_{S}$ also has the properties (1) to (3) and, in addition, it is bounded complete.

We study properties of representations for second-countable Hausdorff spaces
and investigate which property holds under each of the above-mentioned conditions.
Therefore, a space in this paper means a second-countable Hausdorff space unless otherwise noted.
We are mainly interested in the case where $X$ is a regular space because
the corresponding domain representations have good properties as we mentioned above.
In addition, it is proved in \cite{OTY2}  that every second-countable regular 
space has a proper dyadic subbase and in 
\cite{OTY}  that every dense-in-itself second-countable regular space has an independent
subbase,  which is a proper dyadic subbase with an additional property.

We review proper dyadic subbases and their properties in the next section,
and we study TTE-representations and domain representations in $\T^{\omega}$
derived from (proper) dyadic subbases in Section 3.
We introduce the domains ${D}_{S}$ and $\wh{D}_{S}$ in Section 4,
and present the strongly nonadhesiveness condition in Section 5.
Then we study domain representations in these domains for the case $X$ is regular in Section 6.
Finally, in Section 7, we study the small inductive dimension of the $T_{0}$-spaces $L(D_{S})$ and $L(\wh{D}_{S})$ based on a result in \cite{Tsuiki:2004}.

\medbreak
\noindent
{\bf Preliminaries and Notations:}

{\bf Bottomed Sequences:}
Let $\N$ be the set of non-negative integers and $\2$ be the set $\{0,1\}$.
Let $\T$ be the set $\{0,1,\bot\}$
where $\bot$ is called the bottom character which means undefinedness.  
The set of infinite sequences of a set $\Sigma$ is denoted by $\Sigma^\omega$. 
Each element of $\T^\omega$ is called a {\em bottomed sequence} and each copy 
of 0 and
1 which appears in a bottomed sequence $p$ is called a {\em digit} of
$p$.  A {\em finite bottomed sequence} is a bottomed sequence
with a finite number of digits, and the set of all finite bottomed sequences
is denoted by $\T^*$.
We sometimes omit $\bot^{\omega}$ at the end of a finite bottomed sequence
and identify a finite bottomed sequence with a finite sequence of $\T$.
The set of finite sequences of $\2$ is denoted by $\2^{*}$.

We define the partial order relation $\sse$ on $\T$ by
$\bot \sse 0$ and $\bot \sse 1$,
and its product order on $\sigbot$ is denoted by the same symbol $\sse$, i.e., 
for every $p , q\in\sigbot$, $p \sse q$
if $p(n)\sse q(n)$ for each $\nn$.
Then $\2^\omega$ is the set of maximal elements of $\sigbot$.  
We consider the $T_0$-topology
$\{\emptyset, \{0\}, \{1\}, \{0,1\}, \T\}$ on $\T$,
and its product topology on $\sigbot$.
We write $\dom(p)=\{k: p(k) \neq \bot\}$ for $p \in \T^{\omega}$.
For a finite bottomed sequence $e \in \T^*$,
the length $|e|$ of $e$ is the maximal number $n$
such that  $e(n-1) \neq \bot$.
We denote by $p|_{n}$ the finite bottomed sequence
with $\dom(p|_{n}) = \dom(p) \cap \{0,1,\ldots,n-1\}$ 
such that $p|_{n} \sse p$.  That is, 
$p|_n(k) = p(k)$ if $k<n$ and $p|_n(k) = \bot$ if $k \geq n$.  
Note that the notation $p|_n$ is used with a different meaning
in \cite{Tsuiki:2004b}.

The letters $a$ and $b$ will be used for elements of $\2$, 
$c$ for  elements of $\T$, 
$i,j,k,l,m,n$ for elements of $\N$, 
$p$ and $q$ for
bottomed sequences, and $d$ and $e$ for finite bottomed sequences.
We write  $c^\omega=(c,c,\cdots)\in\sigbot$ for $c \in\T$.
We denote by $p[n:=a]$ the bottomed sequence $q$ such that
$q(n) = a$ and $q(i) = p(i)$ for $i \ne n$. 

\medbreak
{\bf Topology:}
Throughout this paper, $X$ denotes a second-countable Hausdorff space
unless otherwise noted.
Therefore, if $X$ is regular,
then $X$ is separable metrizable by Urysohn's metrization theorem.
Recall that a subset $U$ of $X$ is {\em regular open} if $U$
is the interior of its closure. 

A {\em filter} $\fk{F}$ on the space $X$ is a family of subsets of $X$
with the following properties.
\begin{enumerate}
\item $\emptyset \not\in \fk{F}$.
\item If $A\in \fk{F}$ and $A\subseteq A' \subseteq X$, then $A'\in \fk{F}$.
\item If $A,B\in \fk{F}$, then $A\cap B\in \fk{F}$.
\end{enumerate}
Let $\fk{V}(x)$ denote the family of neighbourhoods of $x\in X$.
For a filter $\fk{F}$ on $X$ and a point $x\in X$,
if we have $\fk{V}(x)\subseteq \fk{F}$
then we say that $\fk{F}$ {\em converges to} $x$.

A family $\BB$ of subsets of $X$ is called a {\em filter base}
if it satisfies $\emptyset\not\in \BB$, $\BB\neq \emptyset$, and that
for all $A,B\in \BB$ there exists $C\in \BB$ such that $C\subseteq A\cap B$.
A filter generated by a filter base $\BB$ is defined
as the minimum filter containing $\BB$.
We say that a filter base {\em converges to} $x\in X$
if it generates a filter which converges to $x$.

We denote by $\cl_{Y} A$, $\bd_{Y} A$, and $\ext_{Y} A$  the
closure, boundary, and exterior of a set $A$ in a space $Y$, respectively,
and we omit the subscript if the space is obvious.

\medbreak

{\bf Domain Theory:}
Let $(P,\sse)$ be a partially ordered set (poset).
We say that two elements $p$ and $p'$ of a poset $P$ are {\em compatible}
if $p  \sse q$ and $p' \sse q$ for some $q \in P$,
and write $p \comp p'$ if $p$ and $p'$ are compatible. 
For $p  \in P$ and $A \subseteq P$,
we define $\upa{p}=\{q : q \sqsupseteq p \}$, 
$\down{p}= \{q : q \sse p \}$, 
$\upa{A}= \cup \{\upa q : q \in A \}$, and
$\down{A}= \cup \{\down q : q \in A \}$.
Therefore,  we have $\down{\upa{p}}= \{q: q \comp p \}$.
We say that $A$ is {\em downwards-closed} if $A = \down{A}$,
and {\em upwards-closed} if $A=\upa{A}$.

A subset $A$ of a poset $P$ is called {\em directed}
if it is nonempty and each pair of elements
of $A$ has an upper bound in $A$.  A {\em directed complete partial
order (dcpo)} is a partially ordered set in which  every directed
subset $A$ has a {\em least upper bound} ({\em lub}) $\lub A$.
A dcpo is {\em pointed} if it has a least element.

Let $(D,\sse)$ be a dcpo.
A {\em compact element} of $D$ is an element $d \in D$ 
such that for every directed subset $A$,
if $d \sse \lub A$ then $d \in \down{A}$.
An element of $D$ is called a {\em limit element} if it is not compact.
We write $K(D)$ for the set of compact elements of $D$,
and $L(D)$ for the set of limit elements of $D$.

For $x \in D$, we define $K_x = K(D) \cap \down{x}$.  A dcpo
$D$ is {\em algebraic} if $K_x$ is directed and $\lub K_x= x$ for each
$x \in D$,
and it is {\em $\omega$-algebraic}
if $D$ is algebraic and $K(D)$ is countable.  
In this paper, a {\em domain} means an $\omega$-algebraic pointed dcpo.
The {\em Scott topology} of a domain $D$ is the topology generated by
$\{\upa{d} :d \in K(D)\}$. 
In this paper, we consider a domain $D$ as a topological space with the Scott topology.
A poset is {\em bounded complete} if every subset
which has an upper bound also has a least upper bound.
$\T^{\omega}$ is a bounded complete domain such that $K(\T^\omega)= \T^*$.

An {\em ideal} of a poset $P$ is a directed downwards-closed subset.  
The set of ideals of $P$ ordered by set inclusion is denoted by $\Idl(P)$.
The poset $\Idl(P)$ becomes a domain
called the {\em ideal completion} of $P$ if $P$ is countable.
We have an order isomorphism $K(\Idl(P)) \cong P$ for each countable poset $P$ 
with a least element. 
On the other hand, for a domain $D$, 
we have $\Idl(K(D)) \cong D$.
Therefore, $K(D)$, the set of compact elements of $D$,
determines the structure of $D$.  We say that an ideal of $K(D)$ is 
{\em principal}  if its least upper bound is in $K(D)$.
An infinite strictly increasing sequence
{$d_0 \sqsubsetneq d_1 \sqsubsetneq  d_2  \sqsubsetneq  \dotsb$} 
in {$K(D)$} determines a non-principal ideal
{$\{e \in K(D):e \sse d_i \text{ for some } i\}$}
of {$K(D)$} and thus determines a point of {$L(D)$}.

A poset $P$ is a {\em conditional upper semilattice with least element  (cusl)}
if it has a least element
and every pair of compatible elements has a least upper bound.
 If $P$ is a cusl, then $\Idl(P)$ is a bounded complete domain.
For background material on domains, see \cite{GHKLMS,AJ1994,Sto94}.

\medbreak
{\bf Representation:}
We write $f :\subseteq A \to B$ if $f$ is a partial function from $A$ to $B$.
For a finite or a countably infinite alphabet $\Sigma$,
a surjective partial function from $\Sigma^\omega$ to $X$ 
is called a {\em (TTE-)representation} of  $X$. 
We say that
a continuous function $\gamma:\subseteq \Sigma^{\omega} \to X$ is
{\em reducible} (resp.~ {\em continuously reducible})  to $\delta:\subseteq \Sigma^{\omega} \to X$
if there exists a computable function (resp.~continuous function)
$\phi:\subseteq \Sigma^{\omega} \to \Sigma^{\omega}$
such that $\gamma(p) = \delta( \phi(p))$ for every $p \in \dom(\gamma)$.
Two representations $\delta, \delta':\subseteq \Sigma^{\omega} \to X$ are
{\em equivalent} (resp. {\em continuously equivalent})  if they are reducible (reps.~continuously reducible) to each other.
A representation $\delta:\subseteq \Sigma^{\omega} \to X$
is called {\em admissible}
if $\delta$ is continuous and every continuous function
$\gamma:\subseteq \Sigma^{\omega} \to X$
is continuously reducible to $\delta$.

Let $X$ be a $T_{0}$-space and $B = \{B_{n}: \nn\}$ be a subbase of $X$
indexed by $\N$.  Consider the representation
$\delta_B: \subseteq \N^{\omega} \to X$ such that
$\delta_B(p) = x$ if and only if
$\{p(k) : k \in \N\}  = \{n \in \N: x \in B_{n} \}$.
$\delta_B$ is called
a {\em standard representation} of $X$ with respect to $B$.
Any representation
which is continuously equivalent to a standard representation is
admissible \cite{WG,Schroder2002,Weihrauch}.  

\medbreak
{\bf Domain representation:}
Let $D$ be a domain,  $D^{R}$ a subspace of $D$, 
and $\mu$ a quotient map from $D^{R}$ onto $X$. 
The triple $(D, D^{R}, \mu)$ is called
a {\em domain representation} of $X$.
Note that we do not require $D$ to be bounded-complete 
or each element of $D^{R}$ to be total (i.e., condense)  in this paper.
See \cite{Berger:1993, Blanck:2000} for the notion of totality.
A domain representation
is called a {\em retract domain representation} if $\mu$ is a retraction,
and a {\em homeomorphic domain representation} if $\mu$ is a homeomorphism.

A domain representation  $(D, D^{R}, \mu)$ of $X$ is {\em upwards-closed}
if $D^R$ is upwards-closed and 
$\mu(d)=\mu(e)$ for every $d \spe e \in D^{R}$.
A domain representation $(D, D^{R}, \mu)$ is called {\em admissible}
if for every pair $(E, E^{R})$ of a domain $E$
and a dense subset $E^{R} \subseteq E$
and for every continuous function $\nu:E^{R} \to X$,
there is a continuous function $\phi: E \to D$ such that
$\nu(x) = \mu (\phi (x)) $ holds for all $x \in E^{R}$.  
A domain representation $E = (E, E^{R}, \nu)$ {\em reduces continuously} to
a domain representation $D = (D, D^{R}, \mu)$
if there is a continuous function $\phi:E \to D$ such that 
$\phi(E^{R}) \subseteq D^{R}$ and
$\nu(x) = \mu ( \phi (x))$  for all $ x \in E^{R}$. 
For more about (admissible) domain representations, see
\cite{Blanck:2000,Hamrin,StoTucker95,StoTucker08}.

\section{Proper dyadic subbases} \label{s-2}

Recall that a space means a Hausdorff space unless otherwise noted.

\begin{defi}
A {\em dyadic subbase} $S$ of a space $X$ is a family
$\{S_{n,a}: \nn,  \a2\}$ of regular open sets indexed by $\N \times \2$ such that
(1) $S_{n,1}$ is the exterior of $S_{n,0}$ for each $\nn$ and
(2) it forms a subbase of $X$.
\end{defi}

\noindent
Note that we allow duplications in $S_{n,a}$ and therefore, for example,
a one point set $X = \{x\}$ has a dyadic subbase
$S_{n,0} = X, S_{n,1} = \emptyset \ (n = 0,1,\ldots)$.
Note also that
this definition is applicable also to non-Hausdorff spaces, 
though we only consider the case $X$ is Hausdorff in this paper.
We denote by $S_{n,\bot}$ the set  
$X \setminus (S_{n,0} \cup S_{n,1})$.
Since $S_{n,0}$ is regular open, we get
$\bd S_{n,0} = \bd S_{n,1} =S_{n,\bot}$.
Note that
$S_{n,\bot}$ is defined differently in \cite{Tsuiki:2004b}.

A topological space is called {\em semiregular}
if the family of regular open sets forms a base of $X$.  
It is immediate that a regular space is semiregular.
From the definition, a space with a dyadic subbase is
a second-countable semiregular space.
On the other hand, it is shown in \cite{Tsuiki:2004b} 
that every second-countable semiregular space has a dyadic subbase.

From a dyadic subbase $S$,
we obtain a topological embedding $\varphi_S: X \to \T^\omega$ as follows.
$$
\varphi_{S}(x)(n) = \left\{\begin{array}{ll}
0 &( x \in S_{n,0}),\\
1 &( x \in S_{n,1}),\\
\bot &(x \in S_{n,\bot}).
\end{array}
\right. 
$$
We denote by $\tilde x$ the sequence $\varphi_{S}(x) \in \T^{\omega}$ 
and denote by $\widetilde X$ the set $\varphi_{S}(X) \subseteq \T^{\omega}$
if there is no ambiguity of $S$.

In the sequence $\tilde{x}$,
if $\tilde{x}(n) = a$ for $\nn $ and $\a2$,
then this fact holds for some neighbourhood $A$ of $x$
because $S_{n,a}$ is open.
On the other hand, if $\tilde{x}(n) = \bot$, 
then  every neighbourhood $A$ of $x$ contains points $y_0$ and $y_1$
with $\tilde{y_0}(n) = 0$ and $\tilde{y_1}(n) = 1$.
Therefore, if $\tilde{x}(n) = \bot$,
then every neighbourhood $A$ of $x$ does not exclude
both of the possibilities $\tilde{x}(n) = 0$ and $\tilde{x}(n) = 1$.

\begin{figure}[t]
\begin{center}
\resizebox{13cm}{7cm}{\includegraphics{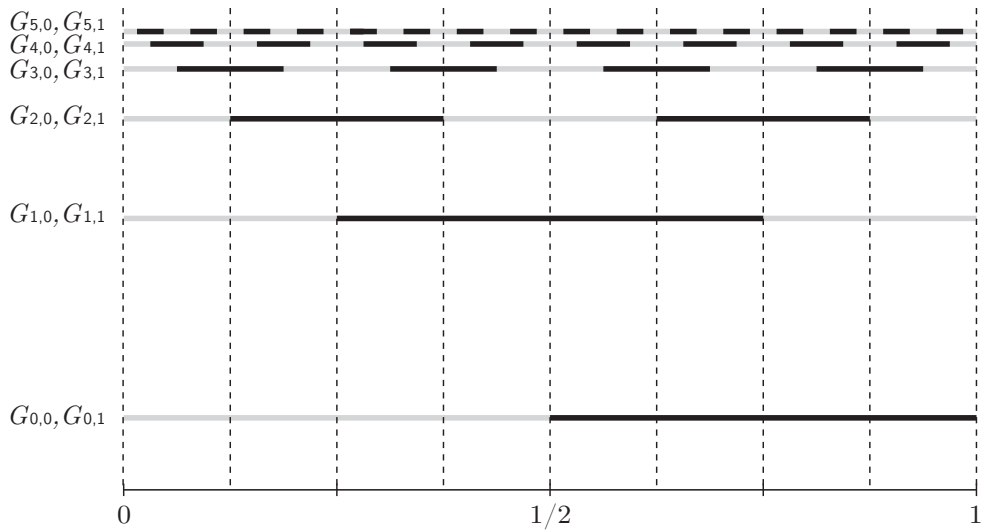}}
\caption{Gray subbase of the unit interval $\I$.\label{fig1}}
\end{center}
\end{figure}

\begin{exa}[Gray subbase]\label{e-1}
Let $\I=[0,1]$ be the unit interval
and let $X_0 = [0, 1/2)$ and $X_1 = (1/2, 1]$ be subsets of $\I$. 
The tent function is the function $t :\I \to\I$ defined as
\[
t(x) = \left\{
\begin{array}{ll}
2 x&(x \in \cl X_0),\\
2 (1- x)\ \ \ \  & (x \in \cl X_1).
\end{array}
\right.
\]
We define the dyadic subbase $G$ as 
$$
G_{n,a} = \{x : t^{n}(x) \in X_{a}\}
$$
for $\nn$ and $\a2$.
The map $\varphi_{G}$ is an embedding of the unit interval in $\T^{\omega}$
\cite{Gianantonio:1999,Tsuiki:2002}.
If $x$ is a dyadic rational number other than $0$ or $1$, then $\varphi_{G}(x)$ has the form
$e \bot 1 0^{\omega}$ for $e \in \2^*$,
and it is in $\2^{\omega}$ otherwise.
Figure \ref{fig1} shows the Gray subbase,
with the gray lines representing $G_{n,0}$
and the black lines representing $G_{n,1}$.
\end{exa}

For a dyadic subbase $S$ and $p \in \T^{\omega}$, let 
\begin{align}
&S(p)=\bigcap_{k \in \dom(p)}S_{k,p(k)},\\
&\bar{S}(p)=\bigcap_{k \in \dom(p)} {\cl S_{k,p(k)}}
= \bigcap_{k \in \dom(p)} ({S_{k,p(k)}} \cup {S_{k,\bot}})
\end{align}
denote the corresponding subsets of $X$.
Note that, for $x\in X$ and $p\in \T^{\omega}$,
\begin{alignat}{2}
x\in S(p) &\Leftrightarrow 
\tilde{x}(k)=p(k)\text{ for }k\in \dom (p)&
&\Leftrightarrow p\sse \tilde{x}
,\label{eq:s}\\
x\in \bar{S}(p) &\Leftrightarrow
\tilde{x}(k)\sqsubseteq p(k)\text{ for } k\in \dom(p)&
&\Leftrightarrow p\comp \tilde{x}.\label{eq:cs}
\end{alignat}

For $e \in \T^{*} = K(\T^{\omega})$, $S(e)$ is an element of 
the base generated by the subbase $S$.
We denote by $\BB_S$  the base 
$\{S(e): e \in \T^*\} \setminus \{\emptyset \}$.
On the other hand, $L(\T^{\omega})$ is the space in which 
$X$ is represented as the following proposition shows.

\begin{prop}
Suppose that $S$ is a dyadic subbase of a space $X$.
\begin{enumerate}
\item $S(\tilde{x})=\{x\}$ for all $x\in X$.
\item $\widetilde{X} \subseteq L(\T^{\omega})$.
\end{enumerate}
\end{prop}
\proof\hfill
\begin{enumerate}
\item Let $x,y$ be distinct elements in $X$.
Since $X$ is $T_1$, there exists $e\in \T^*$ such that
$x\in S(e)$ and $y\not\in S(e)$.
From \eqref{eq:s}, we have $e\sse \tilde{x}$ and $e\not\sse \tilde{y}$.
So we get $\tilde{x}\not\sse \tilde{y}$,
therefore, $y\not\in S(\tilde{x})$.

\item
Suppose that $\dom(\tilde{x})$ is finite.
Then $S(\tilde{x})$ is an open set and thus $\{x\}$ is a clopen set 
which  contradicts the fact that $x$
is on the boundary of $S_{n,a}$ for $n \not\in \dom(\tilde{x})$.
\qed
\end{enumerate}

\begin{defi}
We say that a dyadic subbase $S$ is {\em proper} if $\cl
S(e)=\bar{S}(e)$ for every $e \in \T^*$.  
\end{defi}

If $S$ is a proper dyadic subbase,  then 
$\bar{S}(e)$ is the closure of the base element $S(e)$.
Therefore, by (\ref{eq:cs}),
the sequence $\tilde{x}$ codes not only base elements 
to which $x$ belongs but also base elements to whose closure $x$ belongs.

\begin{prop}
\label{p:proper}
Suppose that $S$ is
a proper dyadic subbase of a space $X$.
\begin{enumerate}
\item
If $x\in X$ and $p \sqsupseteq \tilde{x}$,
then the family $\{S(e) : e \in K_{p}\}$ is a filter base
which converges to $x\in X$.
\item
If $x \neq y \in X$,
then $x \in S_{n,a}$ and $y \in S_{n,1-a}$ for some $\nn$ and $\a2$.
That is, $x$ and $y$ are separated by some subbase element.
\item
If $x\in X$ and $p \sqsupseteq \tilde{x}$, then $\bar{S}(p) = \{x\}$.
\item
If $p \in \2^\omega$, then $\bar{S}(p)$ is either
a one-point set $\{x\}$  for some $x\in X$ or the empty set.
\end{enumerate}
\end{prop}
\proof\hfill
\begin{enumerate}
\item Since we have $\tilde{x}\comp e$ for every $e\in K_{p}$,
we obtain $\cl S(e)=\bar{S}(e)\neq \emptyset$.
Therefore, we get $\emptyset \not\in \{S(e) : e \in K_{p}\}$.

\item Since $X$ is Hausdorff, there is $e \in \T^{*}$ such that
$x \in S(e)$ and $y \not \in \cl{S(e)} = \bar{S}(e)$.
That is, $e\sse \tilde{x}$ and $e\ncomp \tilde{y}$ by (\ref{eq:s}) and (\ref{eq:cs}).
Therefore, 
we get $\tilde{x}\ncomp \tilde{y}$.

\item From (2), 
we have $\bar{S}(\tilde{x}) =\{x\}$.
We get
$\bar{S}(p) \subseteq \bar{S}(\tilde{x})=\{x\}$
from $p \sqsupseteq \tilde{x}$,
and $\bar{S}(p) \ni x$ from $p \comp \tilde{x}$.

\item Let $p \in \2^{\omega}$.
If $p \sqsupseteq \tilde{x}$ for some $x\in X$,
then $\bar{S}(p)$ is a one-point set $\{x\}$ by (3). 
If $p \not\sqsupseteq \tilde{x}$ for all $x\in X$,
then $p\ncomp \tilde{x}$, because $p$ is maximal.
Therefore, $\bar{S}(p)$ is empty.\qed
\end{enumerate}
\cite{Tsuiki:2004b} contains 
an example of a non-proper dyadic subbase for which Proposition \ref{p:proper}
(1) to (4) do not hold.

Finally, we define a property of a dyadic subbase which is stronger than  properness.

\begin{defi}
An {\em independent subbase} is a dyadic subbase such that $S(e)$ is not empty
for every $e \in \T^*$.
\end{defi}

\begin{prop}[\cite{OTY}]\label{ppp1}
An independent subbase is proper.
\end{prop}

The Gray subbase in Example \ref{e-1} is an independent subbase and
we show many independent subbases as examples of
proper dyadic subbases.
When $S$ is an independent subbase, we have
$S(d) \supseteq S(e)$ if and only if $d \sse e$. 
In particular, $S(d) \ne  S(e)$ if $d \ne e$. 
Therefore, for an independent subbase $S$,
the poset $(\BB_S, \supseteq)$ ordered by reverse inclusion
is isomorphic to $\T^{*}$.

\section{Representations and domain representations\\
derived from dyadic subbases}
\label{s-3}

We study some representations and
domain representations of a space $X$
derived from a (proper) dyadic subbase of $X$.  

We introduce two representations.
The first one is immediately derived from a dyadic subbase.
If $S$ is a dyadic subbase of $X$,
then the inverse $\varphi_{S} ^{-1}$ of the embedding $\varphi_{S}$
is a representation of $X$
with the alphabet $\Gamma = \{0,1,\bot \}$
where $\bot$ is considered as an ordinary character of $\Gamma$.  
Each point is represented uniquely with this representation
and it is easy to show that $\varphi_{S}^{-1}: \subseteq \Gamma^{\omega} \to X$ is an
admissible representation
 if and only if $S_{n,\bot} = \emptyset$ for every $n$.

The second one is derived from a proper dyadic subbase.
If $S$ is a proper dyadic subbase of $X$,
from Proposition \ref{p:proper} (3),
we have a map $\rho_S$ from $\upa{\widetilde{X}} \subset \sigbot$ to $X$
such that $\rho_{S}(p)$ is the unique element in $\bar{S}(p)$.
In particular,
from Proposition \ref{p:proper} (4),
$\rho_S$ restricted to the maximal elements $\2^{\omega}$ is
a partial surjective map from $\2^{\omega}$ to $X$,
that is, it is a representation of $X$ which we denote by $\rho_S'$.    

\begin{exa}
For the Gray subbase $G$ of $\I$,
$\rho_{G}'$ is a total function from $\2^{\omega}$ to $\I$
which is called the Gray expansion of $\I$~\cite{Tsuiki:2002}.
$\rho_{G}'$ is equivalent to the binary expansion through simple conversion functions.
\end{exa}

As this example suggests, we consider that $\rho_{S}'$
is a generalization of the binary expansion representation to a proper dyadic subbase $S$.
We study its continuity in Proposition \ref{p:1234}.
It is not admissible in general as the following proposition shows.

\begin{prop}\label{p:adm1}
Suppose that $S$ is a proper dyadic subbase of a space $X$.
$\rho_{S}'$ is admissible
if and only if $S_{n,\bot} = \emptyset$ for every $n$.
\end{prop}
\proof  
Only if part: suppose that $\rho_{S}'$ is admissible and $x \in S_{n,\bot}$.
Theorem 12 of \cite{BrattkaH02} says that every admissible representation
has a continuously equivalent open restriction.
Suppose that
$\delta:\subseteq \2^{\omega} \to X$ is such an open restriction of $\rho_{S}'$ and $x = \delta(p)$.  
Let $a =  p(n)$.
Since $\delta$ is an open map,
$\delta(\{q \in \2^{\omega}: q(n) = a\})$ is an open neighbourhood of $x$,
and since $\delta$ is a restriction of  $\rho_{S}'$,
$\delta(\{q \in \2^{\omega}: q(n) = a\}) \subseteq
\rho_{S}'(\{q \in \2^{\omega}: q(n) = a\}) = S_{n,a} \cup S_{n,\bot}$.
Therefore, $S_{n,a} \cup S_{n,\bot}$ is a neighbourhood of $x$,
which contradicts with $x \in S_{n,\bot}$.

If part:  since the base $\BB_{S}$ is composed of closed and open sets, $X$ is regular
and therefore $\rho_{S}'$ is continuous by Proposition \ref{p:1234} below.
Since  $S_{n,\bot}$ is empty,
$x \in S_{n,0}$ or $x \in S_{n,1}$ holds for every $x \in X$.
Therefore, one can construct a reduction
from the standard representation of $X$ with respect to an enumeration of the subbase $S$
to $\rho'_{S}$.  
\qed

Next, we study domain representations.
We start with a general construction of a domain representation from a base of a space.
Suppose that $\BB$ is a base of a space $X$ such that 
$\emptyset \not\in \BB$, $X \in \BB$, and
$\BB$ is closed under finite non-empty intersection.
For the domain $D_{\BB}$ obtained as the ideal completion of the poset
$(\BB, \supseteq)$ with the reverse inclusion and
for the map $\iota(x) = \{U \in \BB : x \in U\}$ from $X$ to $D_{\BB}$,
$\iota$  is a homeomorphic embedding of $X$ in $D_{\BB}$.
Therefore, $(D_{\BB}, \iota(X), \iota^{-1})$  is a 
homeomorphic domain representation which 
is known to be admissible \cite{Blanck:2000,Hamrin,StoTucker08}.

We introduce two domain representations derived from (proper) dyadic subbases.
The first one is $(\T^{\omega}, \widetilde{X} , \varphi_{S}^{-1})$, which is defined for a space $X$ with a dyadic subbase $S$. 
Since $\varphi_{S}$ is an embedding, it is a 
homeomorphic domain representation.
In particular,
if $S$ is an independent subbase,
then the poset $\BB_S$ is isomorphic to the poset $\T^*$.
Therefore, the domain $D_{\BB_S}$ is isomorphic to $\T^{\omega}$
and thus the domain representations $(\T^{\omega}, \widetilde{X} , \varphi_{S}^{-1})$
and $(D_{\BB_S}, \iota(X), \iota^{-1})$ coincide.
However, if $S$ is a dyadic subbase which is not independent, 
then the poset $\T^*$ provides only a ``notation'' of the base $\BB_S$,
and a set $S(d)$ may be the same as $S(e)$
for $d \ne e \in \T^{*}$.
We show that $(\T^{\omega}, \widetilde{X} , \varphi_{S}^{-1})$ is
an admissible domain representation even for this case.

\begin{prop}\label{admwtX}
If $S$ is a dyadic subbase of a space $X$, then
$(\T^{\omega}, \widetilde{X} , \varphi_{S}^{-1})$
is an admissible domain representation.
\end{prop}
\proof 
Suppose that $E^{R}$ is a subset of a domain $E$
and $\mu$ is a continuous map from $E^{R}$ to $X$.
Define a function $\phi: K(E) \to \T^{\omega}$ as
$\phi(e)(n) = a$ if and only if
$\mu(\upa{e} \cap E^{R}) \subseteq S_{n,a}$.
Since $\phi$ is monotonic, 
it has a continuous extension to $E$, which is a continuous function from
$E$ to $\T^{\omega}$.  It
is also denoted by $\phi$.
We show that the function $\phi$ satisfies
$\varphi^{-1}_{S} (\phi (p)) = \mu(p)$ for $p \in E^{R}$.
We have
$\phi(p)(n) = \lub_{e \in K_{p}} \phi(e)(n)$.
Therefore, for $\a2$, $\phi(p)(n) = a$
if and only if 
$(\exists e \in K_{p}) ( \phi(e)(n) = a) $,
if and only if
$(\exists e \in K_{p}) (\mu(\upa e \cap E^{R}) \subseteq S_{n,a})$,
if and only if
$\mu(p) \in S_{n,a}$.
Therefore, $\phi(p) =\varphi_{S}(\mu(p))$.
\qed

The other domain representation is 
$(\T^\omega, \upa{\widetilde{X}}, \rho_S)$, which is defined
for a regular space $X$ with a proper dyadic subbase $S$.
Suppose that $S$ is a proper dyadic subbase of a space $X$.
From Proposition \ref{p:proper}, 
$\rho_{S}$ is a map from $\upa{\widetilde{X}}$ to $X$.
We have $\varphi_{S}(\rho_{S}(p)) \sse p$ and $\rho_{S}(\varphi_{S}(x)) = x$.
Therefore, 
$(\T^\omega, \upa{\widetilde{X}}, \rho_S)$ is
an upwards-closed retract domain representation
if and only if $\rho_{S}$  is a quotient map.
Blanck showed in Theorem 5.10 of \cite{Blanck:2000} that 
if a topological space has an upwards-closed retract Scott domain representation,
then it is a regular Hausdorff space.    
Therefore, $(\T^\omega, \upa{\widetilde{X}}, \rho_S)$ 
is a domain representation
only if $X$ is regular.  
We show this fact as a corollary to the following equivalence.

\begin{prop}  Let $S$ be a proper dyadic subbase of a space $X$.
The followings are equivalent.
\begin{enumerate}
\item $X$ is regular.
\item $\rho_S: \upa{\widetilde{X}} \to X$ is continuous.
\item $\rho_S' : \subseteq \2^{\omega} \to X$ is continuous.
\end{enumerate}\label{p:1234}
\end{prop}
\proof\hfill
\noindent
\begin{description}
\item[$(1 \Rightarrow 2)$]
Let $p \in \upa{\widetilde{X}}$ and $x = \rho_{S}(p)$.
Since $\{S(\tilde{x}|_{n}) : \nn\}$ is a neighbourhood base of $x$ in $X$
and  
$\{\upa{p|_{m}} \cap \upa{\widetilde{X}} : m \in \N\}$
is a neighbourhood base of $p$ in $\upa{\widetilde{X}}$,
it suffices to show that
for every $n$, there is $m$ such that
$\rho_{S}(\upa{p|_{m}} \cap \upa{\widetilde{X}}) \subseteq S(\tilde{x}|_{n})$.
Since $X$ is regular,
there is $m>n$ such that 
$x \in S(\tilde{x}|_{m}) \subseteq \cl{S(\tilde{x}|_{m})}
\subseteq S(\tilde{x}|_{n})$.
Then, for all $q \in \upa{\widetilde{X}}$
such that $q \sqsupseteq p|_{m}$, 
we have $\varphi_{S}(\rho_{S}(q)) \comp \tilde{x}|_{m}$
because $\varphi_{S}(\rho_{S}(q)) \sse q \sqsupseteq p|_{m}
\sqsupseteq \tilde{x}|_{m}$.
Thus, $\rho_{S}(q) \in \bar{S}(\tilde{x}|_{m})$. 
Therefore,
$\rho_{S}(q) \in \bar{S}(\tilde{x}|_{m})= \cl S(\tilde{x}|_{m})
\subseteq S(\tilde{x}|_{n})$.

\item[$(2 \Rightarrow 3)$]  Immediate.

\item[$(3 \Rightarrow 1)$]  Suppose that $x \in X$ and $\nn$.
For each $p \in \upa{\tilde{x}} \cap \2^{\omega}$, 
since $\rho'_{S}$ is continuous on $p$,
there exists $e_{p}  \in K_{p}$ such that
$\rho'_{S}(\upa{e_p} \cap \2^{\omega}) \subseteq S(\tilde{x}|_{n})$.
It means that $\bar{S}(e_{p}) \subseteq S(\tilde{x}|_{n})$.
Here, we can assume that $\tilde{x}|_{|e_{p}|}  \sqsubseteq e_{p}$
by replacing $e_{p}$ with
$e_{p} \sqcup \tilde{x}|_{|e_{p}|}$.
Note that $\upa{e_p} \cap \2^{\omega}$ for $p \in \upa{\tilde{x}}\cap \2^{\omega}$ 
is an open cover of $\upa{\tilde{x}} \cap \2^{\omega}$
and $\upa{\tilde{x}} \cap \2^{\omega}$ is compact
because it is homeomorphic to $\2^{j}$ for 
some 
$j \leq \omega$. Therefore, 
for some finite subset $\{p_{0},\ldots,p_{h-1}\}$ of $\upa{\tilde{x}} \cap \2^{\omega}$, 
we have $\cup_{i<h}\upa{e_{p_i}} \supseteq \upa{\tilde{x}} \cap \2^{\omega}$.
Let $m$ be the maximal length of $e_{p_{i}}$ for $i < h$
and let $l = m - |\dom(x|_{m})|$. 
Let $d_{0},\ldots, d_{2^{l}-1}$ be sequences of length $m$
obtained by filling the first $l$ bottoms of $\tilde{x}|_{m}$ with 0 and 1.
We have 
$(\upa{d_{0}} \cup \ldots \cup \upa{d_{2^{l}-1}}) \cap \2^{\omega}
= \upa{\tilde{x}|_{m}} \cap \2^{\omega}$.
Therefore, 
$\cup_{i < 2^{l}-1}\bar{S}(d_{i}) =  \bar{S}(\tilde{x}|_{m})$.
On the other hand,
for each $i <2^l$, there is $j < h$ such that
$d_{i} \sqsupseteq e_{p_{j}}$.
Therefore, 
$\bar{S}(d_{i}) \subseteq \bar{S}(e_{p_{j}})\subseteq S(\tilde{x}|_{n})$. 
Thus, we have 
$\bar{S}(\tilde{x}|_{m}) \subseteq S(\tilde{x}|_{n})$.
Since $\bar{S}(\tilde{x}|_{m}) = \cl{ S(\tilde{x}|_{m})}$,
it means that $X$ is a regular space.
\qed
\end{description}

\begin{cor}\label{cor-adm}
Suppose that $S$ is a proper dyadic subbase of a space $X$.
The triple
$(\T^\omega, \upa{\widetilde{X}}, \rho_S)$ is a domain representation
if and only if $X$  is regular.
In this case, it is an admissible retract domain representation.
\end{cor}
\proof
Suppose that $X$ is regular.
From Proposition \ref{p:1234}, $\rho_{S}$ is a retraction with 
right inverse $\varphi_{S}$.  Therefore, $\rho_{S}$ is a quotient map.
Since $\widetilde{X} \subseteq \upa{\widetilde{X}}$
and $\varphi_{S}^{-1}$ is
a restriction of $\rho_{S}$ to $\widetilde{X}$,
the identity map on $\T^{\omega}$ 
is a reduction map from the admissible domain representation
$(\T^\omega, \widetilde{X}, \varphi_S^{-1})$
to $(\T^\omega, \upa{\widetilde{X}} , \rho_S)$.
\qed

\section{Domains $D_{S}$ and $\wh{D}_{S}$}

In the previous section, we studied domain representations in the domain $\T^{\omega}$.
In the following sections, we study domain representations in  subdomains 
$D_{S}$ and $\wh D_{S}$ of $\T^{\omega}$. 
Before that, we consider the domain $E_{S}$ which is defined as the closure
of $\widetilde{X}$ in $\T^{\omega}$. 
It is easy to show that the triple
$(E_{S}, \widetilde{X}, \varphi_{S}^{-1})$ is a dense domain representation of $X$
and, if  in addition $S$ is proper and $X$ is regular,  then
$(E_{S}, \upa{\widetilde{X}}, \rho_{S})$ is a dense admissible retract domain representation of $X$.
In these domain representations, 
we have $S(e) \ne \emptyset$ for every $e \in K(E_{S})$
and the family $\{S(e) : e \in K_{p}\}$ forms a filter base for every $p \in L(E_{S})$.
In this sense, one can say that $E_{S}$ does not contain superfluous elements.
However, $E_{S}$ is identical to $\T^{\omega}$ if $S$ is an independent subbase
and the domain $E_{S}$ does not have information about $X$ in this case.
We consider further restrictions of $\T^{\omega}$ and define the domains $D_{S}$
and $\wh D_{S}$ as follows.

\begin{defi}
Let $S$ be a dyadic subbase of a space $X$.  
\begin{enumerate}
\item We define the poset $K_S \subseteq \T^*$ as
\[
K_S = \{p|_{m} : p \in \widetilde{X} , m \in \N \}
\]
and define $D_{S} = \Idl(K_{S})$.

\item We define the poset $\wh{K}_S \subseteq \T^*$ as
\[
\wh{K}_S = \{p|_{m} : p \in \upa{\widetilde{X}}, m \in \N\}
\]
and define $\wh D_{S} = \Idl(\wh K_{S})$.
\end{enumerate}

\end{defi}
For the Gray-subbase $G$ of $\I$, we have $\bot 1  \in K_{G}$
because $\varphi_{G}(1/2) = \bot 1 0^\omega$,
but $\bot 0 \not \in K_{G}$ and $\bot \bot 1 \not \in K_{G}$.
Figure \ref{fig3} shows the structure of $D_G = \wh{D}_{G}$.

We have $K_{S} \subseteq \wh{K}_{S} \subseteq \T^*$
and $D_{S} \subseteq \wh{D}_{S} \subseteq \T^{\omega}$ for a dyadic subbase $S$.
We also have $\tilde X \subseteq D_{S}$ and $\upa {\tilde X} \subseteq \wh{D}_{S}$.

\begin{figure}[t]
\begin{center}

\resizebox{12cm}{6cm}{\includegraphics{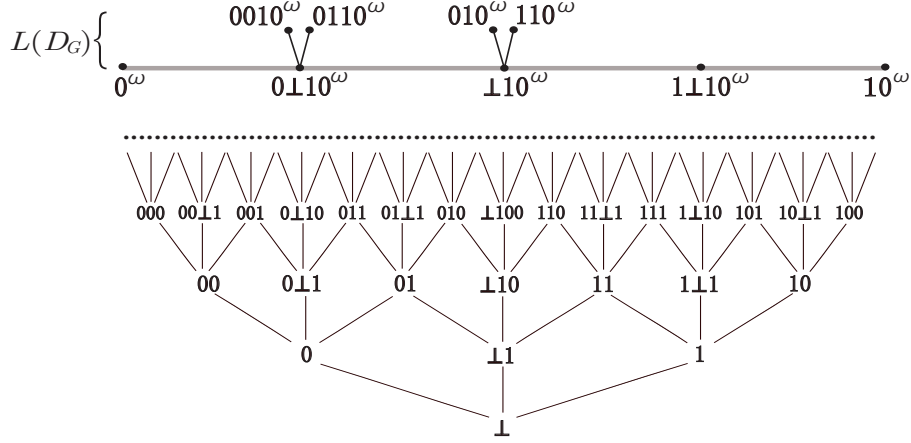}}

\caption{The domain $D_{G}$. \label{fig3}}
\end{center}
\end{figure}

\begin{prop}\label{p:density}\hfill
\begin{enumerate}
\item If $S$ is a dyadic subbase of a space $X$, then $\widetilde{X} $ is dense in $D_{S}$.

\item If $S$ is a proper dyadic subbase of a space $X$, then $\widetilde{X} $ is dense in 
$\wh{D}_{S}$.
\end{enumerate}
\end{prop}
\proof\hfill
\begin{enumerate}
\item 
$\widetilde{X}$
is dense in $D_{S}$ because $S(e)$ is not empty for every $e \in K_{S}$.  

\item By Proposition \ref{p:proper}(1), $S(e)$ is not empty for every  $e \in \wh{K}_{S}$. \qed
\end{enumerate}


The domains $D_{S}$ and $\wh{D}_S$ are not equal in general
as the following example shows.

\begin{exa}\label{ex:height}
Let $W$ be the space obtained by glueing four copies of $\I$ at 
one of the endpoints.
That is, $W = \2 \times \2 \times \I /\!\! \sim$ for $\sim$ the equivalence
relation identifying $(a, b, 0)$ for $a, b \in \2$.  Let $z$ be the identified point.
That is, $z = [(a, b, 0)]$ for $a, b \in \2$.
Let $ R$ be the dyadic subbase defined as 
\begin{align*}
 R_{0,c} &= \{c\} \times \2 \times (0,1] / \sim, \\
 R_{1,c} &= \2 \times \{c\} \times (0,1]  / \sim, \\
 R_{n+2,c} &= \2 \times \2 \times G_{n,c} / \sim,
\end{align*}
for $\nn$ and $c \in \2$.
We have $\tilde z = \bot \bot 0^{\omega} \in L(D_{R})$ and
$a b  0^{\omega} \in L(D_{R})$ for $a, b \in \2$.
However,  $ a \bot 0^{\omega} \not \in L(D_{R})$ for $\a2$
and $ \bot b 0^{\omega} \not \in L(D_{R})$ for $b \in \2$.
On the other hand, we have 
$ a \bot 0^{\omega} \in L(\wh D_{R})$ for $\a2$
and $ \bot b 0^{\omega} \in L(\wh D_{R})$ for $b \in \2$.
\end{exa}

\begin{prop}\label{p:lub}
If $S$ is a dyadic subbase of a space $X$, then 
$\wh{K}_S$ is a cusl and therefore
$\wh{D}_S$ is a bounded complete domain.
\end{prop}
\proof
Let  $d = p |_{m}, e = q |_{n}  \in \wh{K}_{S}$
for $p, q \in \upa{\widetilde{X}} $ and $m \leq n \in \N$. 
Suppose that  $d \comp e$ in $\wh{K}_{S}$
and let $f = d \lub e$ be their least upper bound in $\T^*$.
Then, since $d$ and $q|_{n}$ are compatible and
$|d| \leq n$, $d \comp q$ in $\T^{\omega}$ and
$r = d \lub q$ satisfies $r|_{n} = f$.
Since $r  \sqsupseteq q \sqsupseteq \tilde{x}$
for some $x \in X$,
we have $f \in \wh{K}_{S}$.
\qed

\begin{prop}\hfill\label{p:admDS}
\begin{itemize}
\item If $S$ is a dyadic subbase of a space $X$,  then
the domain representation $(\wh{D}_S, \widetilde{X}, \varphi_{S}^{-1})$ is
admissible.

\item If $S$ is a proper dyadic subbase of a regular space $X$,
then the domain representation 
$(\wh{D}_S, \upa{\widetilde{X}}, \rho_{S})$ is admissible.
\end{itemize}
\end{prop}
\proof\hfill
\begin{enumerate}
\item We show that there is a reduction
from the admissible domain representation
$(\T^{\omega},  \widetilde{X}, \varphi_{S}^{-1})$ to 
$(\wh{D}_S, \widetilde{X}, \varphi_{S}^{-1})$.
Since $\wh{D}_S$ is bounded complete,
we define $\phi : \T^{\omega} \to \wh{D}_S$ as
$\phi(p) =  \lub\{e \in \wh{K}_S: e \sse p\}$.
It preserves $\widetilde X$
because $\{S(e) : e \in \wh{K}_S, e \sse \tilde x\}$
contains $S(\tilde x|_n)$ for every $n$.

\item  The map $\phi$ preserves $\upa{\widetilde{X}}$
and it is a reduction also from
$(\T^{\omega},  \upa{\widetilde{X}}, \rho_{S})$ to 
$(\wh{D}_S, \upa{\widetilde{X}}, \rho_{S})$. 
\qed
\end{enumerate}

As Example \ref{e:2} shows, ${D}_S$ is not bounded complete in general.
It is left open whether the results corresponding to Proposition \ref{p:admDS} 
 hold for $D_{S}$.

\begin{exa}\label{e:2}
Let $Y$ be the space obtained by glueing $1/4$ and $3/4$ in $\I$.
That is,
$Y$ is the quotient space of $\I$
with the equivalence relation generated by $1/4 \sim 3/4$.
Let $T$ be the independent subbase of $Y$ such that 
$T_{0,0} = (G_{0,0} \setminus \{1/4\})/\!\sim$,
$T_{0,1} = (G_{0,1} \setminus \{3/4\})/\!\sim$, 
and $T_{n,a} = G_{n,a}/\!\sim$ for $n > 0$.
We have $\varphi_{T}(z) = \bot \bot 1 0^{\omega}$ 
for $z = [1/4] = [3/4]$ 
and $\varphi_{T}([x]) = \varphi_{G}(x)$
for $x \not \in \{1/4, 3/4\}$.
Therefore, 
$K_{S}$ contains $\bot \bot 1$ and $\bot 1 (= \varphi_{T}([1/2])|_{2})$,
which are bounded above by $0 1 1 = \varphi_{T}([1/3])|_{3}$
and $1 1 1 = \varphi_{T}([2/3])|_{3}$.
However, $\bot 1 1 $,
which is the least upper bound of $\bot\bot 1$ and $\bot 1$
in $\T^{\omega}$,
does not belong to $K_{T}$.
Therefore, $K_{T}$ is not a cusl and $D_{T}$ is not a bounded complete domain.
Note that the poset $\wh{K}_{T}$ contains $\bot11$
because
$\bot 1 1 0^{\omega} \sqsupseteq  \varphi_{T}(z)$.
\end{exa}

In Example \ref{e:2}, $\upa{\varphi_{T}(z)}$ in $D_{T}$
is the set 
$\{\bot \bot p, 0 0  p,  0 1  p,  1 0  p,  1 1  p,  1 \bot p, 0 \bot p\}$
for $p = 1 0^{\omega}$.
Therefore, it is different from 
$\upa{\varphi}_{T}(z)$ in $\T^{\omega}$
which contains also $\bot 0 p$ and
$\bot 1 p$.  
As Example \ref{e:2} and \ref{ex:height} show, 
$L(D_{S}) \subsetneq L(\wh{D}_{S})$ in general.
However, for a proper dyadic subbase $S$, 
$D_{S}$ and $\wh{D}_{S}$ coincide on the top elements 
as Proposition \ref{l:01} shows.

 \begin{lem} \label{lemmale}
Let $S$ be a dyadic subbase of a space $X$.
\begin{enumerate}
\item For $p \in D_S$ and $\nn$, we have $p|_{n} \in K_S$. 
\item For $p \in \wh{D}_S$ and $\nn$, we have $p|_{n} \in \wh{K}_S$. 
\end{enumerate}
 \end{lem}
 \proof\hfill
\begin{enumerate}
\item  Suppose that $p$ is the least upper bound of an ideal
$\{\tilde{x_{i}}|_{m_i}  :  i \in I\}$.  Then, $p|_{n}$ is the least upper bound
of the ideal $\{\tilde{x_{i}}|_{m_i}|_{n}  :  i \in I\} = 
\{\tilde{x_{i}}|_{\min\{m_{i}, n\}}  :  i \in I\}$, whose length is no more than $n$.
\item It is proved similarly to (1).
 \qed
\end{enumerate}

\begin{prop}\label{l:01}
For a proper dyadic subbase $S$ of a space $X$,
$D_{S} \cap \2^{\omega} = \wh{D}_{S} \cap \2^{\omega}$.   
\end{prop}
\proof
Let $p \in \wh{D}_{S} \cap \2^{\omega}$ and $n \in \N$.
By Lemma \ref{lemmale}, we have $p|_{n} \in \wh{K}_{S}$.
Therefore, $S(p|_{n}) \ne \emptyset$ by Proposition \ref{p:density}(2).
For every $y \in S(p|_{n})$, 
we have $p|_{n} = \tilde y|_{n}$ since $p \in \2^{\omega}$. 
Thus,  $p|_{n} \in K_{S}$ for every $\nn$ and 
we have $p \in D_{S}$.

\qed

\section{Domains with minimal-limit sets}

We study structures of $D_{S}$ and $\wh{D}_{S}$ and present a condition
on $X$ which ensures the existence of minimal elements of 
$L(D_{S})$ and $L(\wh{D}_{S})$.

\begin{defi}
Let $P$ be a poset. 
\begin{enumerate}
\item $x \in P$ is a {\em minimal element} if $y
\sse x$ implies $y = x$ for all $y \in P$.
We write $\min(P)$ for the set of all minimal elements of $P$.
\item We say that $P$ has {\em enough minimal elements}
if, for all $y \in P$,
there exists $x \in \min(P)$ such that $x \sse y$.
\item For a domain $D$, 
if $L(D)$ has enough minimal elements, we call
$\min(L(D))$ the {\em minimal-limit set} of $D$.
\end{enumerate}
\end{defi}

\noindent The poset $L(\T^{\omega})$ does not have enough minimal elements.

\begin{defi}\hfill
\begin{enumerate}
\item[(1)] Let $(P, \sqsubseteq)$ be a pointed poset with the least element $\bot_{P}$.
The {\em level} of $x \in P$, if it exists,  is the maximal length $n$ of a chain 
$\bot_{P} = y_{0} \sqsubsetneq y_{1} \sqsubsetneq \ldots \sqsubsetneq y_{n} = x$,
and it is denoted by $\level(x)$.
\item[(2)] A poset $P$ is {\em stratified} if it is pointed and every element of $P$ has a level.
\item[(3)] We say that $y$ is an {\em immediate successor} of $x$
if $x \sqsubsetneq y$ and there is no element $z$ such that
$x \sqsubsetneq z \sqsubsetneq y$.
We write $\succ (x)$ for the set of
immediate successors of $x$.
\item[(4)] We say that a stratified poset $P$ is {\em finite-branching}
if $\succ(x)$ is finite for every $x \in P$.  
\end{enumerate}
\end{defi}

In \cite{Tsuiki:2004}, the following proposition is proved with a slightly stronger
definition of finite-branchingness that contains  
the condition $\level(y) = \level(x) + 1$ for $y \in \succ(x)$. 
However, one can check that this condition is not used in the proof and
it holds with our definition of finite-branchingness.
\begin{prop}[Proposition 4.13 of \cite{Tsuiki:2004}]
\label{prop:fb}
If $D$ is a domain such that $K(D)$ is  finite-branching,
then $L(D)$ has enough minimal elements
and $\min(L(D))$ is compact. \qed
\end{prop}

\begin{defi}\hfill\label{def-ad}
\begin{enumerate}
\item
We say that a space $X$ is {\em adhesive}
if $X$ has at least two points and
closures of any two non-empty open sets
have non-empty intersection.
\item We say that $X$ is {\em nonadhesive}
if it is not adhesive.
\item We say that $X$ is {\em strongly nonadhesive}
if every open subspace is nonadhesive.
\end{enumerate}
\end{defi}

Nonadhesiveness (and even strongly nonadhesiveness)
is a weak condition that
many of the Hausdorff spaces satisfy. 
A space is called Urysohn (or completely Hausdorff or ${T}_{2\frac{1}{2}}$ in some literature) 
if any two distinct points can be separated by closed neighbourhoods.
A regular space is always Urysohn.
 
\begin{prop}\label{prop:urysohn}
Every Urysohn 
space is strongly nonadhesive. \qed
 \end{prop}

Note that there is an adhesive Hausdorff space as the following example shows.

\begin{exa}\label{exampleA}
Let $P$ be the set of dyadic irrational numbers in $\I = [0,1]$
and $\N^{+}$ be the set of positive integers.
We define our space $A = P \cup \N^{+}$.
A neighbourhood base of $x \in P$ is $U \cap P$ for $U$
a Euclidean neighbourhood of $x \in \I$.
A neighbourhood base of $n \in \N^{+}$ is the union of $\{n\}$ and 
$U \cap { P}$ for $U$ a Euclidean 
neighbourhood of $\{k/2^n : \text{$k$ is an odd number}, 0 < k < 2^{n}\}$. 
One can easily verify that $A$ is Hausdorff.
The closure of $U \cap P$ is 
$(U \cap P) \cup \{n \in \N^{+}: k/2^{n} \in U
\text{ for some odd number $k$}\}$ and it contains $\{n \in \N^{+}: n > m\}$ for some $m$.
Therefore $A$ is adhesive.
The space $A$ has the following independent subbase $S^{A}$.
\begin{align*}
S^{A}_{n,a} = (G_{n,a} \cap P) \cup
\{n\in \N^{+} : k/2^{n} \in G_{n,a} \text{ for all odd number $k < 2^{n}$ }\}.
\end{align*}
We have 
$\varphi_{S^{A}}(x) = \varphi_{G}(x)$  for $x \in P$, and
$\varphi_{S^{A}}(n) = \bot^{n} 1 0^{\omega}$ for $n \in \N^{+}$. 
\end{exa}

As Propositions \ref{p-ad} and \ref{p-3} show,
adhesiveness of $X$
and finite-branchingness of $D_{S}$ are closely related.
Recall again that a (non)adhesive space means a (non)adhesive Hausdorff space.

\begin{prop}\label{p-ad}
Suppose that $X$ is an adhesive space
and $S$ is a proper dyadic subbase of $X$ such that
$S_{n,a} \ne \emptyset$ for every $\nn$ and $\a2$.
Then, $\succ(\bot^{\omega})$ in $K_{S}$ is infinite.
Therefore,  $K_{S}$ is not finite-branching.
\end{prop}
\proof
All the elements of $\succ(\bot^{\omega})$
have the form $\bot^{k}a$ for $k \in \N$ and $\a2$.
Suppose that $\succ(\bot^{\omega})$ is finite
and let $\bot^{n-1}a$ be an element with the maximal length.
For $b = 1-a$, take $x \in S_{n,b}$ and $p \in \upa {\tilde x} \cap \2^{\omega}$. 
For $d = p|_{n}$, $d \in \wh K_{S}$ holds and therefore
$S(d) \ne \emptyset$ by  Proposition \ref{p:density}.
Let $e \in \2^{n}$ be the bitwise complement of $d$.
Since $\bot^{n-1}a \in K_{S}$, $e \in \wh K_{S}$ and therefore
$S(e) \ne \emptyset$ by  Proposition \ref{p:density}.
Therefore, closures of $S(d)$ and $S(e)$ intersect.
Since $S$ is proper, $\cl S(d) = \bar{S}(d)$
and $\cl S(e) = \bar{S}(e)$.
Therefore, there exists $y \in \bar{S}(d) \cap \bar{S}(e)$.
Since $\tilde{y}|_{n} = \bot^{\omega}$, 
the smallest index of digits in $\tilde{y}$ is greater than $n$,
and we have contradiction.
\qed

For the independent subbase $S^{A}$ of $A$ 
in Example \ref{exampleA},
$\succ(\bot^{\omega}) = \{\bot^{k}1 : k \in \N\}$.

\begin{prop}\label{p-3}
Suppose that $X$ is a nonadhesive space and
$S$ is a proper dyadic subbase of $X$.  Then, $\succ(\bot^{\omega})$ in $K_{S}$ is finite. 
\end{prop}
\proof
Since $S$ is nonadhesive,
for some $p, q \in \T^*$, $S(p) \ne \emptyset$,
$S(q) \ne \emptyset$,
and $\cl{S(p)} \cap \cl{S(q)} = \emptyset$ hold.
Since $S$ is proper, 
$\bar{S}(p) \cap \bar{S}(q) = \emptyset$.
Let $n = \max\{|p|, |q|\}$.  If $\tilde x |_{n}  = \bot^{\omega}$ for some $x \in X$, then 
$x \in \bar{S}(p)$ and $x \in \bar{S}(q) $ and we have contradiction.
Thus, in $K_{S}$,
$\succ(\bot^{\omega}) \subseteq
\{\bot^{k} a \bot^{\omega} : k  < n, \a2\} $. 
\qed

\begin{lem}\label{l-4}
Suppose that $S$ is a proper dyadic subbase of a space $X$ and 
$e \in \T^*$.  Let $\nu$ be an enumeration of $\N \setminus \dom(e)$.
Then, $$T_{n,a} = S_{\nu(n),a} \cap S(e) \ \ \ (\nn , \a2)$$
is a proper dyadic subbase of $S(e)$.\qed
\end{lem}
\begin{proof}
Let $A$ be the regular open set $S(e)$.
First, note that if $P$ is a regular open subset of $X$, then 
$A\cap P$ is a regular open subset of $A$  and  $\ext_{A} (A \cap P) = A \cap \ext_{X} P$.
Therefore, $T$ is a dyadic subbase.
Note also that $\cl_{A}(A \cap P) = A \cap \cl_{X} P$.
Therefore, for $d \in \T^{*}$, we have
\begin{align*}
\cl_{A} \bigcap _{k \in \dom(d)} T_{k,d(k)}  
&= \cl_{A} (A \cap \bigcap _{k \in \dom(d)} S_{\nu(k),d(k)})\\
&= A \cap  \cl_{X}\bigcap _{k \in \dom(d)} S_{\nu(k),d(k)} \\
&= A \cap  \bigcap _{k \in \dom(d)} \cl_{X} S_{\nu(k),d(k)}\\
&= \bigcap _{k \in \dom(d)} (A \cap \cl_{X} S_{\nu(k),d(k)})\\
&= \bigcap _{k \in \dom(d)} \cl_{A}(A \cap S_{\nu(k),d(k)})\\
&= \bigcap _{k \in \dom(d)} \cl_{A} T_{k,d(k)}.
\end{align*}
Therefore, $T$ is proper.
\end{proof}

\begin{prop} \label{prop:fb-proper}
Suppose that $X$ is a strongly nonadhesive space and
$S$ is a proper dyadic subbase of $X$. 
\begin{enumerate}
\item The poset $K_{S}$ is finite-branching.
\item  The poset $\wh K_{S}$ is finite-branching.
\end{enumerate}
\end{prop}
\proof \hfill
\begin{enumerate}
\item Let $e \in \T^*$.
By applying Proposition \ref{p-3}
to the proper dyadic subbase $T$ on $S(e)$ in Lemma \ref{l-4},
$\succ(\bot)$ is finite in the poset $K_{T}$.
Since $K_{T}$ is identical to $\upa{e}$ in $K_{S}$,
$\succ(e)$ is finite in $K_{S}$.
\item  
In this proof, $\succ(d)$ for $d \in K_{S}$ means $\succ(d)$ in 
$K_{S}$.  Let $e \in \wh{K}_{S}$ and let
$k$ be the maximal length of elements in $\cup_{d \in \down\,\, e \cap K_{S}} \succ(d)$,
which exists by (1).
Suppose that, for  some $n \geq  k$ and $\a2$,  $e[n:=a] \in \wh K_{S}$.
Then, for some $x \in X$ and $p \sqsupseteq \tilde x$, 
$p|_{n+1} = e[n:=a]$.  Therefore, $\tilde x|_{n} \sqsubseteq e$. 
For $d_{0} = \tilde x|_{n}$, 
let $m > n$ be the least integer such that $\tilde x|_{m} \sqsupsetneq d_{0}$.
The set $\succ(d_{0})$ contains $\tilde x|_{m}$
 and we have contradiction.
\qed
\end{enumerate}

\begin{thm} \label{theorem:min}
Suppose that $X$ is a strongly nonadhesive space and
$S$ is a proper dyadic subbase of $X$.
\begin{enumerate}
\item $L(D_S)$ has enough minimal elements and
$\min(L(D_S))$ is compact.
\item $L(\wh{D}_S)$ has enough minimal elements
and $\min(L(\wh{D}_S))$ is compact.
\end{enumerate}
\end{thm}
\proof
From Proposition \ref{prop:fb} and \ref{prop:fb-proper}.\qed

Note that, as Proposition \ref{prop:urysohn} shows,
Theorem \ref{theorem:min}  
is
applicable to all the Urysohn spaces, in particular, to regular spaces.
Note also that the premise of Theorem \ref{theorem:min} is
not a necessary condition for $L(D_{S})$
to have enough minimal elements.
For example, for the space $A$
and the dyadic subbase $S^{{A}}$
in Example \ref{exampleA},
the domain $D_{S^{ A}}$ has enough minimal elements
and $\varphi_{S^{ A}}( A)
\subseteq \min(L(D_{S^{ A}}))$.

It is shown in \cite{Tsukam-Tsuiki-Kokyuroku} that
there is a Hausdorff space $X$
and an independent subbase $S$ of $X$ such that
$D_{S}$ is equal to $\T^{\omega}$
and therefore $L(D_{S})$ does not have
enough minimal elements.

\section{Domain representations in minimal-limit sets}

Now, we show that $X$ is embedded in $\min (L(D_{S}))$ and $\min (L(\wh{D}_{S}))$
for the case $S$ is a proper dyadic subbase of a regular space $X$.
We start with new notations 
and a small lemma.

\begin{defi}
For a dyadic subbase $S$ of a space $X$,
$p \in \T^{\omega}$, and $\nn$,
we define $\exS^{n}(p) \subseteq X$
and $\exbarS^{n}(p) \subseteq X$ as follows.
\begin{alignat*}{2}
\exS^{n}(p) &= \bigcap_{k < n}S_{k,p(k)} &
 &= S(p|_{n}) \ \cap
\bigcap_{\scriptsize\begin{array}{c}
k < n,\\ k \not \in \dom(p)\end{array}}
S_{k,\bot},\\
\exbarS^n(p) &= \bigcap_{k < n}\cl S_{k,p(k)} &
\ &= \bar{S}(p|_{n}) \ \cap
\bigcap_{\scriptsize\begin{array}{c}
k < n,\\ k \not \in \dom(p)\end{array}}
S_{k,\bot}.\\
\end{alignat*}
  \end{defi}

\begin{lem}\label{lemexs}
  Let $e \in \T^{*}$ and $n = |e|$.
\begin{enumerate}
\item $\exS^{n}(e) \ne \emptyset$ if and only if $e \in K_{S}$.
\item $\exbarS^{n}(e) \ne \emptyset$ if and only if $e \in \wh{K}_{S}$.
\end{enumerate}
\end{lem}
\proof\hfill
\begin{enumerate}
\item $x \in \exS^{n}(e)$ if and only if $e = \tilde x|_{n}$.

\item $x \in \exbarS^{n}(e)$ if and only if $\tilde x|_{n} \sqsubseteq e$ if and only if
there exists $q \sqsupseteq \tilde x$ such that $e = q|_{n}$.
\qed
\end{enumerate}

\begin{thm}\label{t-main}
Suppose that $S$ is a proper dyadic subbase of a regular space $X$
 and $D \in \{D_{S}, \wh{D}_S\}$.
If $p \in L(D)$ and $p$ is compatible with $\tilde{x}$ in $\T^{\omega}$, then $p \spe \tilde{x}$.
In particular, 
$\widetilde{X}  \subseteq \min(L(D)).$
\end{thm}
\proof
Suppose that $p \in L(\T^{\omega})$ satisfies $p \comp \tilde{x}$ and $p \not \spe \tilde{x}$.  There is an index $\nn$ such that $\tilde{x}(n) \neq \bot$ and $p(n) = \bot$.  We
assume that $\tilde{x}(n) = 0$.  That is, $x \in S_{n,0}$.  
Since $X$ is regular and $S$ is proper, $x \in S(e) \subseteq \cl S(e) = \bar{S}(e)
 \subseteq S_{n,0}$
for some $e \in K(D_{S})$.
We can assume that $e = \tilde{x}|_{m}$ for some $m > n$
such that $p(m-1) \ne \bot$.

We have
\begin{align*}
\bar{S}(e) &= \bigcap_{k \in \dom(e)} ({S_{k,e(k)} \cup  S_{k,\bot}}),  \\
\exbarS^m(p) &= 
\bigcap_{\scriptsize\begin{array}{c}k < m,\\ k \in \dom(p)\end{array}}
({S_{k,p(k)} \cup  S_{k,\bot}}) \cap
\bigcap_{\scriptsize\begin{array}{c}k < m,\\ k \not \in \dom(p)\end{array}}
S_{k,\bot}.\\
\end{align*}
Therefore, since $p \comp e$,
we have $S_{n,0} \supseteq \bar{S}(e) \supseteq \exbarS^{m}(p) \supseteq \exS^{m}(p)$. 
On the other hand, since $p(n) = \bot$,  we have $S_{n,0} \cap \exbarS^{m}(p) = \emptyset$.  Therefore,
we can conclude that both $\exbarS^m(p) = \exbarS^m(p|_{m})$
and $\exS^m(p) = \exS^m(p|_{m})$  are empty.
Thus, by Lemma \ref{lemexs},  we have
$p|_{m} \not \in K_{S}$ and $p|_{m} \not \in \wh{K}_{S}$.
Then, from Lemma \ref{lemmale}, we have $p\not  \in D_{S}$ and $p \not \in \wh{D}_{S}$.
\qed

\begin{thm}\label{propact}
Suppose that $S$ is a proper dyadic subbase of a
compact Hausdorff space $X$ 
 and $D \in \{D_{S}, \wh{D}_S\}$.
We have $\widetilde{X}  = \min(L(D))$ and $X$ is a retract of $L(D)$.
\end{thm}

\begin{proof}
Since a compact Hausdorff space is regular, 
we have $\widetilde{X}  \subseteq \min(L(D))$ by Theorem \ref{t-main}.
Assume that there exists $p\in \min(L(D)) \setminus \widetilde{X}$.
For every $x \in X$, $\tilde{x}$ and $p$ are not compatible in $\T^{\omega}$ by Theorem \ref{t-main}.
Therefore, $\tilde{x}(k)$ and $p(k)$ are different digits for some $k$.
Thus, we have an open covering 
$X=\bigcup_{k\in \dom(p)} S_{k,1-p(k)}.$
Since $\bar{S}(p|_m)\neq \emptyset$ for all $m \in \N$,
there is no finite subcovering. Therefore, $X$ is not compact.
\end{proof}

We state properties of domain representations as a corollary.
Here, $\upa{\widetilde{X}}$ in $(D_{S}, \upa{\widetilde{X}},
\rho_{S})$ is the upwards-closure of $\widetilde{X}$ in $D_{S}$ which
may be different from the upwards-closure of $\widetilde{X}$ in
$\T^{\omega}$.

\begin{cor}\label{maincor}
Suppose that $S$ is a proper dyadic subbase of a regular space $X$
and $D \in \{D_{S}, \wh{D}_S\}$.
 \begin{enumerate}
\item In the dense domain representation $(D, \tilde X, \varphi^{-1})$, we have $\tilde X \subseteq \min(L(D))$.  
In particular, if $X$ is
compact,  then $\tilde X = \min(L(D))$.  

\item In the retract domain representation $(D, \upa{\widetilde{X}}, \rho_{S})$, $\upa{\widetilde X}$ is downwards-closed in $L(D)$.
In particular, if $X$ is
compact,  then $\upa{\widetilde X} = L(D)$.   
\end{enumerate}
\end{cor}
\begin{proof}
By Theorem \ref{t-main} and \ref{propact}.
\end{proof}

As Corollary \ref{maincor} shows, if $X$ is compact, then
$(D_{S}, L(D_{S}), \rho_{S})$ and $(\wh D_{S}, L(\wh D_{S}), \rho_{S})$ are
representations of $X$ as minimal-limit sets of domains 
studied in \cite{Tsuiki:2004}. 
In both of the domains, all the strictly increasing sequences in the set of compact elements  denote points of $X$ via $\rho_{S}$.

As we have seen, if $S$ is a proper dyadic subbase of a regular space $X$, 
then $\min(L(D_{S}))$ is a compact space in which $X$ is embedded densely.
Therefore, $\min(L(D_{S}))$ is a kind of compactification of $X$.
However, it is not a Hausdorff compactification, in general, as Example \ref{e5} shows.

\begin{exa}\label{e5}
Let $Z = \I \times \I$ be a unit square. 
An independent subbase $H$ of $Z$ is defined as
\[
H_{2k, a} = G_{k,a} \times \I,\quad 
H_{2k+1, a} = \I \times G_{k,a},\quad \text{ for }k \in \N,\, \a2.
\]
We have $\varphi_{H}((1/2, 1/2)) = \bot \bot 1 1 0^{\omega}$.
We set $A = \{0 0 p, 0 1 p, 1 0 p, 1 1 p\}$ and  
$B = \{0 \bot p, 1 \bot p, \bot 0 p, \linebreak \bot 1 p\}$
where $p=110^{\omega}\in \T^{\omega}$.
Note that
$\upa{\bot \bot p} = \{\bot \bot p\} \cup A \cup B$.

Let $Z^{0} = \I \times \I \setminus \{(1/2, 1/2)\}$ be a subspace of $Z$.
The independent subbase of $Z^{0}$ which is obtained 
by restricting each element of $H$ to $Z^{0}$
is denoted by $H^0$.
We have $L(D_{H^{0}}) = L(D_{H}) \setminus \{\bot \bot p\}$
and we get
\[
\min(L(D_{H^{0}})) = (\min(L(D_{H})) \setminus \{\bot \bot p\}) \cup B.
\]
Since the set $B$ contains a pair of compatible bottomed sequences,
$\min(L(D_{H^0}))$ is not Hausdorff.
\end{exa}

\begin{exa}\label{e7}
Let $Z^{1} = Z\backslash\{(1/2,1/3)\}$
be a subspace of $Z$ and $H^1$ be a dyadic subbase of $Z^1$
defined similarly to Example \ref{e5}.
We have
\[
\widetilde{Z}^1=\min(L(D_H))\setminus\{q\}
\]
where $q=\bot 011(01)^{\omega}$.
However, since we have
\[
\left|y-\frac{1}{3}\right|< \frac{1}{3\cdot 2^n}\Rightarrow
\varphi_{H^1}((1/2,y))|_{2n}=q|_{2n}
\]
for all $\nn$,
we get $L(D_{H^1})=L(D_H)$.
Therefore $\min(L(D_{H^1}))$ is a Hausdorff compactification of $Z^1$.
\end{exa}

\begin{exa}\label{e8}
We set $Z^{2} = Z^1\cup\{x_0,x_1\}$
with $\tilde{x}_a=q[0:=a]$ for $q$ in Example \ref{e7} and $\a2$, and
let $H^{2}$ be the corresponding dyadic subbase.
The space $Z^2$ is 
a non-regular Hausdorff space
and we have 
$\min(L(D_{H^2})) = \min(L(D_H))$.
Since we have
$\tilde{x}_a \not\in \min(L(D_{H^2}))$,
we get $\widetilde{X}\not\subseteq\min(L(D_{H^2}))$.
\end{exa}

\section{Height of \texorpdfstring{$L(D_{S})$}{L(DsubS)} and the dimension of \texorpdfstring{$X$}{X}}

We finally study relations between the degree of a proper dyadic subbase $S$
and structures of $L(D_{S})$ and $L(\wh{D}_{S})$.

\begin{defi}
For a dyadic subbase $S$ of a space $X$ and $x \in X$, we define
$\deg_{S}(x) = |\{\nn :  x \in S_{n,\bot}\}|$ and
$\deg S = \sup \{\deg_{S}(x) : x \in X \}$.
\end{defi}
\noindent
If $\deg{S} = m$, then $\varphi_{S}(x)$ contains at most $m$ copies of
$\bot$ for $x \in X$. 
It is proved in \cite{OTY2} that 
 every  separable metrizable space $X$ with $\dim X = m$ 
has a proper dyadic subbase $S$ with $\deg S = m$.  
Here, $\dim  X$ is the covering dimension of $X$.  It is known that 
$\dim X$ is equal to the small inductive dimension $\ind X$ of $X$ for a separable metrizable space $X$.
See, for example,  \cite{engelking1977general} and \cite{engelking1995theory} for dimension theory.  

For a domain $D$, we consider
the small inductive dimension $\ind L(D)$ of the topological space $L(D)$ 
with the subspace topology of the Scott topology of $D$.
In Theorem 6.11 of \cite{Tsuiki:2004},  it is proved that
if $D$ is a domain with property M,  then 
$\ind L(D) = \height L(D)$ holds.    Here, 
$\height P$ is the maximal length of a chain 
$ a_0 \sqsubsetneq a_1 \sqsubsetneq \ldots \sqsubsetneq a_n$ in a poset $P$.  
Property M is defined as follows.

\begin{defi}\label{def-enough}
(1) We say that a poset $P$ is {\em mub-complete} if for every finite subset 
$A \subseteq P$, the set of upper bounds of $A$ has enough minimal elements. 
That is, if $p$ is an upper bound of $A$, then there exists a
minimal upper bound $q$ of $A$ such that $q \sse p$.\\
(2) We say that a domain $D$ has {\em 
property M}
if $K(D)$ is mub-complete and each finite subset $A \subseteq K(D)$ has a 
finite set of minimal upper bounds.
\end{defi}
\noindent
Property M is equivalent to Lawson-compactness
for $\omega$-algebraic dcpo by the 2/3 SFP Theorem
\cite{plotkin81postgraduate}. Domains with property M are studied in \cite{Jung89}.

\begin{prop}\label{p:final}
Suppose that $S$ is a proper dyadic subbase of a regular space $X$.
\begin{enumerate}
\item The domains $D_{S}$ and $\wh{D}_{S}$ have property M.
\item 
$\ind L(\wh{D}_{S}) = \height L(\wh{D}_{S}) \geq 
\ind L({D}_{S}) = \height L({D}_{S}) \geq \dim X$.
\item If $X$ is compact, then $\ind L(\wh{D}_{S}) =\deg S$.
\end{enumerate}
\end{prop}
\proof\hfill
\begin{enumerate}
\item Since bounded completeness implies property M,  $\wh{D}_{S}$
has property M.  For $D_{S}$, suppose that a finite subset $A \subseteq K_{S}$ has
an upper bound. Let $d$ be the least upper bound of $A$ in $\T^{*}$.
Then, $e \in K_{S}$ is an upper bound of $A$ in $K_{S}$ if and only if $e \sqsupseteq d$.
If $d \in K_S$, then it is the only minimal upper bound of $A$.
Suppose that  $d \not \in K_S$, $e \sqsupseteq d$, and $e \in K_{S}$.
Then, for $n = |d|$, $e|_{n} \sqsupseteq d$ and $e|_{n} \in K_{S}$ by Lemma \ref{lemmale}(1).   
 Therefore, if $e$ is a minimal upper bound of $A$, then  $e = e|_{n}$
 and the length of $e$ is no more than $n$.
Therefore, the set of minimal upper bounds of $A$ is finite.

\item The equation $\ind L(D) = \height L(D)$
for $D \in  \{D_{S}, \wh{D}_{S}\}$ 
is derived from (1) and Theorem 6.11 of \cite{Tsuiki:2004}.
We have $\ind L(D_{S}) \geq \ind X$ because
$X$ is embedded in $L(D_{S})$ and $\ind X = \dim X$ for a separable metrizable space $X$.

\item Since $\upa{\tilde{x}}$ in $\wh{D}_{S}$ and
$\upa{\tilde{x}}$ in $\T^{\omega}$ are the same set for $x \in X$,
the maximum number of bottoms in $\tilde x$ for $x \in X$ is equal to 
the height of $L(\wh{D}_{S})$,
which is equal to the small inductive dimension of $L(\wh{D}_{S})$ by (2).
\qed
\end{enumerate}

\noindent
Note that $\ind L(D_{S})$ may not be equal to $\deg{S}$
even for an independent subbase of a compact space $X$ as Example \ref{ex:height}
shows. 
In this example, 
the height of $L(D_{R})$  is one, whereas that of $L(\wh{D}_{R})$ is two.

\section*{Acknowledgement}

The authors thank anonymous referees for  careful readings and many valuable
comments.


\begin{thebibliography}{key}

\bibitem{AJ1994}
S.~Abramsky and A.~Jung.
\newblock Domain theory.
\newblock In {\em Handbook of Logic in Computer Science}, volume~3, pages
  1--158. Oxford University Press, 1994.

\bibitem{Berger:1993}
U.~Berger.
\newblock Total sets and objects in domain theory.
\newblock {\em Ann. Pure Appl. Logic}, 60(2):91--117, 1993.

\bibitem{Blanck:2000}
J.~Blanck.
\newblock Domain representations of topological spaces.
\newblock {\em Theoretical Computer Science}, 247:229 -- 255, 2000.

\bibitem{BrattkaH02}
V.~Brattka and P.~Hertling.
\newblock Topological properties of real number representations.
\newblock {\em Theoretical Computer Science}, 284(2):241--257, 2002.

\bibitem{engelking1977general}
R.~Engelking.
\newblock {\em General topology}.
\newblock Monografie matematyczne. PWN, 1977.

\bibitem{engelking1995theory}
R.~Engelking.
\newblock {\em Theory of dimensions, finite and infinite}.
\newblock Sigma series in pure mathematics. Heldermann Verlag, 1995.

\bibitem{Gianantonio:1999}
P.~D. Gianantonio.
\newblock An abstract data type for real numbers.
\newblock {\em Theoritical Computer Science}, 221:295--326, 1999.

\bibitem{GHKLMS}
G.~Gierz, K.~H. Hofmann, K.~Keimel, J.~D. Lawson, M.~W. Mislove, and D.~S.
  Scott.
\newblock {\em Continuous Lattices and Domains}.
\newblock Cambridge University Press, 2003.

\bibitem{Hamrin}
G.~Hamrin.
\newblock {\em Effective Domains and Admissible Domain Representations}.
\newblock {Ph.D.} thesis, {U}ppsala dissertations in mathematics 42, Uppsala
  University, 2005.

\bibitem{Jung89}
A.~Jung.
\newblock {\em Cartesian Closed Categories of Domains}, volume~66 of {\em CWI
  Tracts}.
\newblock Centrum voor Wiskunde en Informatica, Amsterdam, 1989.

\bibitem{OTY2}
H.~Ohta, H.~Tsuiki, and K.~Yamada.
\newblock Every second-countable regular space has a proper dyadic subbase,
  2013.
\newblock http://arxiv.org/abs/1305.3393.

\bibitem{OTY}
H.~Ohta, H.~Tsuiki, and S.~Yamada.
\newblock Independent subbases and non-redundant codings of separable
  metrizable spaces.
\newblock {\em Topology and its applications}, 158:1--14, 2011.

\bibitem{plotkin81postgraduate}
G.~D. Plotkin.
\newblock Post-graduate lecture notes in advanced domain theory (incorporating
  the ``{P}isa {N}otes'').
\newblock Technical report, Department of Computer Science, University of
  Edinburgh, 1981.

\bibitem{Schroder2002}
M.~Schr{\"o}der.
\newblock Extended admissibility.
\newblock {\em Theoretical Computer Science}, 284(2):519--538, 2002.

\bibitem{Sto94}
V.~Stoltenberg-Hansen, I.~Lindstr{\"o}m, and E.~Griffor.
\newblock {\em Mathematical Theory of Domains}.
\newblock Cambridge University Press, 1994.

\bibitem{StoTucker95}
V.~Stoltenberg-Hansen and J.~Tucker.
\newblock Effective algebra.
\newblock In S.~Abramsky et~al., editors, {\em Handbook of Logic in Computer
  Science}, volume~4, pages 357--526. Oxford University Press, 1995.

\bibitem{StoTucker08}
V.~Stoltenberg-Hansen and J.~V. Tucker.
\newblock Computability on topological spaces via domain representations.
\newblock In {\em New Computational Paradigms}, pages 153--194. Springer, 2008.

\bibitem{Tsuiki:2002}
H.~Tsuiki.
\newblock Real number computation through gray code embedding.
\newblock {\em Theoretical Computer Science}, 284(2):467--485, 1977.

\bibitem{Tsuiki:2004}
H.~Tsuiki.
\newblock Compact metric spaces as minimal-limit sets in domains of bottomed
  sequences.
\newblock {\em Mathematical Structures in Computer Science}, 14(6):853--878,
  2004.

\bibitem{Tsuiki:2004b}
H.~Tsuiki.
\newblock Dyadic subbases and efficiency properties of the induced
  $\{0,1,\bot\}^\omega$-representations.
\newblock {\em Topology Proceedings}, 28(2):673--687, 2004.

\bibitem{Tsukam-Tsuiki-Kokyuroku}
Y.~Tsukamoto and H.~Tsuiki.
\newblock Properties of domain representations of spaces through dyadic subbases 
\newblock {\em Mathematical Structures in Computer Science}, to appear.

\bibitem{Weihrauch}
K.~Weihrauch.
\newblock {\em Computable analysis: an introduction}.
\newblock Springer-Verlag New York, Inc., Secaucus, NJ, USA, 2000.

\bibitem{WG}
K.~Weihrauch and T.~Grubba.
\newblock Elementary computable topology.
\newblock {\em Journal of Universal Computer Science}, 15(6):1381--1422, 2009.

\end{thebibliography}

\end{document}